\documentclass[12pt,leqno]{amsart}

 \newtheorem{thm}{Theorem}[section]
 
 \newtheorem{lem}[thm]{Lemma}
 \newtheorem{prop}[thm]{Proposition}
 \theoremstyle{definition}
 
 \theoremstyle{remark}

 \numberwithin{equation}{section}
 \newtheorem*{conj}{Conjecture}

\begin{document}

\title{On  conciseness of words in profinite groups}
\author{Eloisa Detomi}
\address{Dipartimento di Matematica, Universit\`a di Padova, \\
 Via Trieste 63,\\ 35121 Padova, \\ Italy}
\email{detomi@math.unipd.it}
\author{Marta Morigi}
\address{Dipartimento di Matematica, Universit\`a di Bologna,\\
Piazza di Porta San Donato 5, \\ 40126 Bologna, \\ Italy}
\email{marta.morigi@unibo.it}
\author{Pavel Shumyatsky}
\address{Department of Mathematics, University of Brasilia,\\
Brasilia-DF, \\ 70910-900 Brazil}
\email{pavel@unb.br}

\thanks{This research was partially supported by Universit\`a 
di Padova (Progetto di Ricerca di Ateneo:``Invariable generation of groups").}

%\date{July 14, 2015}

\begin{abstract}
 Let $w$ be a group word. It is conjectured that if $w$ has only countably many values in a profinite group $G$,
 then the verbal subgroup $w(G)$ is finite. In the present paper we confirm the conjecture in the cases 
 where $w$ is a multilinear commutator word,  or the word $x^2$, or the word $[x^2,y]$. 
\end{abstract}

%%% ----------------------------------------------------------------------
\maketitle

\section{Introduction} 

Let $w=w(x_1, \ldots ,x_k)$ be a group-word, and let $G$   be a group. The verbal subgroup $w(G)$ of $G$ determined by $w$ 
  is the subgroup generated by the set of all values $w(g_1,\ldots ,g_k),$ where $g_1,\ldots ,g_k$ are elements of $G.$
 A word $w$   is said to be concise if whenever the set of its values is finite in $G,$ it always follows that the subgroup $w(G)$ is finite.
 More generally, a word $w$ is said to be concise in a class of groups $X$   if whenever the set of its values is finite in a group $G \in X,$
 it always follows that $w(G)$ is finite. P. Hall asked whether every word is concise, but later Ivanov proved that this problem has a negative solution in its general form \cite{ivanov}
 (see also \cite[p.\ 439]{ols}). On the other hand, many relevant words are known to be concise.

For instance, it is an easy observation by P. Hall
 that every non-com\-mu\-ta\-tor word is concise (see e.g. \cite[Lemma 4.27]{rob2v1}).
 A word $w$ is non-com\-mu\-ta\-tor if the sum of exponents of at least one variable involved in $w$ is non-zero.
 It was shown in \cite{jwilson}  that the multilinear commutator words are concise (see also \cite{m+g}). Such words are also known under the name of outer commutator words and are precisely the words that can be written in the form of multilinear Lie monomials. Merzlyakov showed that every word is concise in the class of linear groups 
\cite{merzlyakov}  while Turner-Smith proved that every word is concise in the class of residually finite groups all of whose quotients are again residually finite   \cite{TS2}. It was shown in \cite{as} that if $w$ is a multilinear commutator word and $n$ is a prime-power, then the word $w^n$ is concise in the class of residually finite groups. Another interesting family of words that are concise in residually finite groups was exhibited in \cite{g+s}.

There is an open problem, due to Jaikin-Zapirain \cite{jaikin}, whether every word is concise in the class of profinite groups. Of course, the verbal subgroup $w(G)$ in a profinite group G is defined as the closed subgroup generated by all $w$-values. In the present paper we deal with a newly discovered phenomenon that suggests that the very definition of conciseness in profinite groups can perhaps be relaxed. More precisely, we provide evidence in support of the following conjecture.

\begin{conj} %\label{A}
 Assume that the word $w$ has only countable many values in a profinite group $G.$ Then the verbal subgroup $w(G)$ is finite.
\end{conj}

We show that the above conjecture holds true whenever $w$ is multilinear commutator word.
Thus, our first result is as follows.

\begin{thm}\label{mcw}
 Let $w$ be a multilinear commutator word and $G$ a profinite group having only countably many $w$-values. Then the verbal subgroup $w(G)$ is finite.
\end{thm}

Further, we attempt to deal with the conjecture in the case where $w$ is a non-commutator word. 
Recall that the ``classical" conciseness of such words is just an easy observation. In sharp contrast, 
the above conjecture %\ref{A}
 for non-commutator words seems hard to deal with. So far we succeeded only in the case of the word $w=x^2.$
\begin{thm}\label{squares}
Let $G$ be a profinite group having only countably many squares. Then $G^2$ is finite.
\end{thm}

Here, as usual, $G^n$ denotes the closed subgroup of $G$ generated by $n$-th powers.  

We were also able to confirm the conjecture %\ref{A}
 for the commutator word $w=[x^2,y]$.

\begin{thm}\label{[x^2,y]}
Let $G$ be a profinite group having only countably many values of the word $w=[x^2,y]$. Then the verbal subgroup $w(G)$  is finite.
\end{thm}

% In the context of profinite groups all the usual concepts of group theory are interpreted topologically. In particular, by a subgroup of a profinite group we mean a closed subgroup. A subgroup is said to be generated by a set $S$ if it is topologically generated by $S$.
Throughout the paper  by a subgroup of a profinite group we mean a closed subgroup and we say that a subgroup  is 
 generated by a set $S$ to mean that  it is topologically generated by $S$.

\section{Multilinear commutator words}

A multilinear commutator word (outer commutator word) is a  word which is obtained by nesting commutators, but using always different variables. Thus the word $[[x_1,x_2],[x_3,x_4,x_5],x_6]$ is a multilinear commutator while the Engel word $[x_1,x_2,x_2,x_2]$ is not.  
An important family of multilinear commutator words is formed by the derived words $\delta_k$, on $2^k$ variables, which are defined recursively by
$$\delta_0=x_1,\qquad \delta_k=[\delta_{k-1}(x_1,\ldots,x_{2^{k-1}}),\delta_{k-1}(x_{2^{k-1}+1},\ldots,x_{2^k})].$$
Of course $\delta_k(G)=G^{(k)}$, the $k$-th derived subgroup of $G$. 
Another distinguished family of multilinear commutators are 
the simple commutators $\gamma_k$, given by $$\gamma_1=x_1, \qquad \gamma_k=[\gamma_{k-1},x_k]=
[x_1,\ldots,x_k].$$
The corresponding verbal subgroups $\gamma_k(G)$ are the terms of the lower central series of $G$. 

Recall that a group is periodic (torsion) if every element of the group has finite order % and a profinite group 
 and  a group is called locally finite if each of its finitely generated subgroups is finite. 
 Periodic profinite groups have received a good deal of attention in the past. In particular, using Wilson's reduction theorem  \cite{wilson-torsion}, Zelmanov has been able to prove local finiteness of periodic profinite groups \cite{z}. 
A profinite group is said to be of finite rank $r$ if each subgroup of $G$ can be generated by at most $r$ elements.
In view of the above it is clear that a periodic profinite group of finite rank is finite. 

In this section we will make use of the following  two results from \cite{DMS-coverings}. 

\begin{thm}\label{cov1} Let $w$ be a multilinear commutator word and $G$ a profinite group that has countably
many periodic subgroups whose union contains all $w$-values in $G$. Then $w(G)$ is locally finite.
\end{thm}
\begin{thm}\label{cov2} Let $w$ be a multilinear commutator word and $G$ a profinite group that has countably
many subgroups of finite rank whose union contains all $w$-values in $G$. Then $w(G)$ has finite rank.
\end{thm}

It will be convenient to first prove Theorem \ref{mcw} 
 in the special case when $w=\delta_k$ is a derived word.
 We require a lemma.

\begin{lem}\label{boundedly-many}  Let $m \ge 1$ and $k \ge 0$ be integers. 
% $m, k$ be positive natural numbers.
 There exists a number $t=t(m,k)$ depending on $m$ and $k$ only, such that if $G$ is an
  $m$-generated profinite group, then
 every $\delta_{k}$-value in elements of $G'$ is a product of at most $t$ elements which are $\delta_{k+1}$-values in elements of $G$.
\end{lem}
\begin{proof}
The proof is by induction on  $k$. Let $H$ be the derived subgroup of $G$.  As $G$ is finitely generated,  
a theorem of Nikolov and Segal \cite{NS} tells us that $H$ coincides with the abstract subgroup of $G$ generated by commutators. Moreover every element of $H$ is a product
of $r$ commutators where $r$ depends only on  $m$. 
So for $k=0$ the result follows. 

 Assume that $k\geq 1$ and let $u$ be a $\delta_k$-value in elements of $H$. 
Write $u=[h_1,h_2]$ where $h_1, h_2$ are   $\delta_{k-1}$-values in  elements of $H$. 
  By induction, $h_1,h_2$  are  both  product of at most $t_1=t(m, k-1)$  $\delta_{k}$-values in elements of $G$. 
 Using the well-known commutator identities $[xy,z]=[x,z]^y\,[y,z]$, $[x,yz]=[x,z]\,[x,y]^z$ we can decompose
% each   $\delta_{k}$-value of $H$
 $u$  as product of at most $t_1^2$ commutators $[a,b]$ where $a$ and $b$ are $\delta_{k}$-values in elements of $G$. 
 The lemma follows.
\end{proof}

\begin{prop}\label{soluble}
 Let $G$ be a profinite group and let $k$ be an integer. Suppose that $G$ has only countably many $\delta_k$-values.
 Then $G^{(k)}$ is finite.
\end{prop}

\begin{proof}
 The proof is by induction on $k$. Of course the result holds when $k=0$ because infinite profinite groups 
 are uncountable. 

By assumption, all  $\delta_k$-values are covered by countably many procyclic subgroups. Thus, by Theorem \ref{cov2},   
 $G^{(k)}$  has finite rank. 
 If every    $\delta_k$-value has finite order, then by  Theorem  \ref{cov1},  
 $G^{(k)}$ is locally finite, and hence finite. 
So we now assume that there exist  elements $a_1, \ldots , a_{2^k}$ such that  $\delta_k(a_1, \ldots , a_{2^k})$ has infinite order. Without loss of generality we can assume that $G=\langle a_1, \ldots , a_{2^k} \rangle$. 

 As $G$ is finitely generated, by Lemma \ref{boundedly-many}, every  $\delta_{k-1}$-value of $H=G'$ is a product of boundedly many $\delta_{k}$-values in elements of $G$. Therefore $H$ has only countably many $\delta_{k-1}$-values and by induction 
   $H^{(k-1)}$ is finite. In particular $\delta_k(a_1, \ldots , a_{2^k})$ has finite order, a contradiction. 
\end{proof}

We will require combinatorial techniques developed in \cite{DMS-coverings}.

Let $n\geq 1$. Denote by $I$ the set of all $n$-tuples $(i_1,\dots,i_n)$, where all 
entries $i_k$ are non-negative integers. We will view $I$ as a partially ordered set with the partial 
order given by the rule that $$(i_1,\dots,i_n)\leq(j_1,\dots,j_n)$$ if and only if 
$i_1\leq j_1,\dots,i_n\leq j_n$.

Let $G$ be a group and  $w=w(x_1,\dots,x_n)$  a  multilinear commutator word. 
For every $\mathbf{i}=(i_1,\ldots,i_n) \in I$, we write 
%. Given a  multilinear commutator word  $w=w(x_1,\dots,x_n)$ and  
% $\mathbf{i}=(i_1,\ldots,i_n) \in I$, we write
\[
w(\mathbf{i})=w(G^{(i_1)},\ldots,G^{(i_n)})
\]
 for  the subgroup generated by the $w$-values $w(a_1,\dots,a_n)$ with  $a_{j} \in G^{(i_j)}$. 
Further, set 
\[
w(\mathbf{i^+})=\prod w(\mathbf{j} ),
\]
 where the product is taken over all $\mathbf{j} \in I $ such that $\mathbf{j}>\mathbf{i}$.

%The following two results are Corollary 6 and  a special case of Proposition 7 of \cite{DMS-coverings}, respectively.

%\begin{lem}\label{wi_ab}
% If $ w(\mathbf{i^+})=1$, then $ w(\mathbf{i})$  is abelian.  
%\end{lem}

The following result is  a special case of Proposition 7 of \cite{DMS-coverings}.

 \begin{lem}
\label{skew congruence}
Assume that there exists 
 $\mathbf{i} \in I$ with the property that $ w(\mathbf{i^+})=1$. 
Let $k$ be an index such that $i_k= \max \{ i_j \mid j=1,\ldots,n\}.$
Choose  $a_j\in G^{(i_j)}$ for every $j=1,\ldots,n.$ 
 Then for every integer $s$ we have 
$$w(a_1,\ldots,a_n)^s=w(a_1,\ldots, a_{ k-1},a_{ k}^s,a_{{ k}+1},\ldots, a_{n}).$$
\end{lem}

Now we are ready to prove  Theorem \ref{mcw}.

\begin{proof}[\bf{Proof of Theorem \ref{mcw}}]
Recall that $w$ is a multilinear commutator word and $G$ is a profinite group having only countably many $w$-values.
 We wish to prove that  $w(G)$ is finite.

Let $n$ be the number of variables involved in $w$. As every $\delta_n$-value is a $w$-value 
(see for example \cite[Lemma 4.1]{S}), it follows from Proposition \ref{soluble} that $G^{(n)}$ is finite.  
  Factoring out $G^{(n)}$ we can assume that $G$ is soluble. Thus, there exist only finitely many 
$\mathbf{i} \in I$ such that $ w(\mathbf{i}) \neq 1$. The theorem will be proved by induction on the number of 
such tuples $\mathbf{i}$.
Choose ${\bf i}=(i_1,\dots,i_n)\in I$ such that $w({\bf i})\neq1$ and  $ w(\mathbf{i^+})=1.$ 
% It follows from  Lemma \ref{wi_ab} that  $w({\bf i})$ is abelian. Moreover, 
  Applying Lemma \ref{skew congruence}, we 
 see that if $a_j\in G^{(i_j)}$ for every $j=1,\ldots,n,$ then every integral power
 of $w(a_1,\ldots,a_n)$ is again a $w$-value. As the set of $w$-values is closed, it follows that the 
 procyclic subgroup generated by $ w(a_1,\ldots,a_n)$ consists of $w$-values. Therefore the procyclic subgroup is countable, and hence finite.  
  This proves that $w(a_1,\ldots,a_n)$
 has finite order whenever $a_j\in G^{(i_j)}$ for every $j=1,\ldots,n$. Let $\tilde w$ be the word on $r=2^{i_1}+2^{i_2}+\ldots +2^{i_n}$ variables defined
 by: $$\tilde w(x_1,\dots,x_r)=w(\delta_{i_1}(x_1,\dots x_{2^{i_1}}),\ldots,\delta_{i_n}(x_{r-2^{i_n}+1},\ldots,
 x_{r})).$$
 Then $w(\mathbf{i})=\tilde w(G)$. Moreover, every $\tilde w$-value is in particular an element of the 
 form  $w(a_1,\ldots,a_n)$, where $a_j\in G^{(i_j)}$ for every $j=1,\ldots,n$,
 so it has finite order. Thus the set of $\tilde w$-values is contained in the union of countably many
 finite cyclic subgroups, and  it follows from Theorems \ref{cov1} and \ref{cov2} that $\tilde w(G)$ is both locally finite and 
 of finite rank, so it is finite. We can pass to the quotient $G/\tilde w(G).$ The number of tuples
$\mathbf{i}$ such that $ w(\mathbf{i}) \neq 1$ for the group $G/\tilde w(G)$ is strictly smaller than that for $G.$ 
Hence by induction we conclude that  $w(G)$ is finite, as required.
 \end{proof}
 
 \section{The words $x^2$ and $[x^2,y]$}

We first prove that if the word $x^2$ has only countably many values in a profinite group $G$ then $G^2$ is 
 finite.

\begin{lem}\label{conjclass}
  Let $G$ be a profinite group having a countable conjugacy class $X$. 
%and $X$ a conjugacy class in $G$. If $|X|$ is countable, then $X$ is finite.
 Then the class $X$ is finite. 
\end{lem}

\begin{proof}
 Let $X=\{x_1,x_2,\dots\}$. For every $k \ge 1$  let $S_k=\{y\in G;\ x_1^y=x_k\}$. The sets $S_k$ are closed and they cover 
 the group $G$. By Baire Category theorem \cite[p. 200]{ke} at least one of the sets $S_k$ has non-empty 
 interior. Hence $C_G(x_1)$ is open. 
\end{proof}

\begin{lem}\label{ab}
Let $G$ be an abelian profinite group in which the set $\{x^n|\ x\in G\}$ is countable. Then $G^n$ is finite.
\end{lem}

\begin{proof} The result follows from the fact that, since $G$ is abelian, every element in $G^n$ is an  
$n$-th power and from the fact that a countable profinite group is finite. 
\end{proof}

\begin{lem}\label{lf} Let $G$ be a profinite group in which the set $\{x^n|\ x\in G\}$ is countable. 
Then $G$ is locally finite.
\end{lem}

\begin{proof} It follows from the previous lemma that every procyclic subgroup of $G$ is finite. So $G$ 
is periodic. By  Zelmanov's result \cite{z}, $G$ is   locally finite. 
% The fact that periodic compact groups are locally finite is a famous result by Zelmanov \cite{Z}.
\end{proof}

%\begin{thm}
%Let $G$ be a profinite group in which the set $X=\{x^2|\ x\in G\}$ is countable. Then $G^2$ is finite.
% \end{thm}
 \begin{proof}[\bf{Proof of Theorem \ref{squares}}]
Let $G$ be a profinite group having only countably many squares and  let $X=\{x_1,x_2,\dots\}$ be the set of squares of elements of $G$.
 For every $k \ge 1$ let $S_k=\{y\in G|\ y^2=x_k\}$. The sets $S_k$ cover $G$ so by Baire 
Category Theorem at least one of them has non-empty interior. Therefore $G$ contains a normal open 
subgroup $H$ and an element $a$ such that $(ah)^2=a^2$ for any $h\in H$. By Lemma \ref{conjclass}, 
$a^2$ has only a finite number of conjugates; moreover, by Lemma \ref{lf},  $G$ is locally finite
 % it follows that
 so $a^2$ is contained in a finite normal subgroup of $G$. Thus, we pass to the quotient over that subgroup and
assume that $(ah)^2=1$ for any $h\in H$.
So $a$ inverts  all  elements in $H$ and thus $H$ is abelian. It follows from  Lemma \ref{ab}  that $H^2$ is finite.  
 Passing to the quotient $G/H^2,$ we
can assume that $H$ has exponent $2$.

Now fix an arbitrary $b\in G$. The set $\{(bh)^2|\ h\in H\}$ is countable. 
On the other hand, $(bh_1)^2=(bh_2)^2$ if and only if  $h_1h_2\in C_H(b)$. Therefore the quotient $H/C_H(b)$ is
countable. 
%Since  $C_H(b)$ is a normal closed subgroup in $H$, we 
 We conclude that $C_H(b)$ has finite index  in $H$. This happens for every $b\in G$. Let $b_1H,\ldots,b_nH$ be  the cosets of $H$ in $G$.
As $H$ is abelian, the center  of $G$ contains the intersection of the subgroups $C_H(b_i)$ for $i=1, \dots, n$. 
We conclude that $G$ is central-by-finite. By Schur's Theorem \cite[10.1.4]{rob}, $G'$ is finite. 
 We can pass to the quotient $G/G'$. 
 The result now  follows from Lemma \ref{ab}.
 \end{proof}

The last part of this section is devoted to proving Theorem \ref{[x^2,y]}. 

\begin{lem}\label{free}
 Let $w=[x^2,y]$. There exists a constant $c$ such that in any group $G$ every $\gamma_3$-value 
can be written as the product of at most $c$ values of the word $w.$ % of finitely many $w$-values.
\end{lem}
\begin{proof}
Let $F$ be the free abstract group on free generators $x_1,x_2,x_3$
  and let $K=w(F)$. Since  $F/K$ is nilpotent of class $2$, the commutator $[x_1,x_2,x_3]$ lies in $K$.
  Therefore $[x_1,x_2,x_3]$ can be written as the product of finitely many, say $c$,  $w$-values or 
  their inverses. We remark that in any group $[a^2,b]$ is conjugate to $[a^2,b^{-1}]$.  Therefore the
  inverse of a $w$-value is again a $w$-value. Thus $[x_1,x_2,x_3]$ is the product of at most $c$ $w$-values.
 This holds in the free group and hence in any group.
\end{proof}

  \begin{proof}[\bf{Proof of Theorem \ref{[x^2,y]}}]
%Recall that now $w=[x^2,y]$ and let
Let $G$ be a profinite group having only countably many values of the word $w=[x^2,y]$. By Lemma 
\ref{free}, every    $\gamma_3$-commutator
  is a product of finitely many $w$-values. Hence    $G$ has only countably many    $\gamma_3$-commutators. 
  It follows from  Theorem \ref{mcw} that   $\gamma_3(G)$ is finite. Passing to the quotient $G/\gamma_3(G),$ 
  without loss of generality we can 
  assume that $G' \leq Z(G)$, where $Z(G)$ is the center of $G$. 

Let $H=G^2$. Of course $w(G)=[H,G].$ Since $G/Z(G)$ is abelian, every element of $H$ is a square modulo $Z(G)$. Thus
for every $h\in H$ and $g \in G$ the commutator $[h,g]$ 
is a $w$-value. In particular, $H$ contains only countably many commutators $[h_1,h_2].$ %, with $h_1,h_2\in H$
 It follows from Theorem \ref{mcw} that $H'$  is finite. Passing to the quotient $G/H',$ we
can assume that $H$ is abelian. 

Therefore for every $g\in G$
the subgroup $[H,g]$ coincides with the set of commutators $\{[h,g] \mid h\in H\}$.  
Thus $[H,g]$  is countable, hence finite. 
Since the set of $w$-values in $G$ is countable, % it follows that
 $G$ contains only countably
many finite subgroups generated by $w$-values. In particular $G$ contains countably many finite subgroups
$H_1,H_2,\dots$ such that for every $g\in G$ there exists $i$ with the property that $[H,g]=H_i$.

 For every $k \ge 1$ set $S_k=\{y\in G|\ [H,y]=H_k\}$. By   Baire 
Category Theorem, there exist a positive integer $j$, an open normal subgroup $N$ and an element $g\in G$ 
such that 
$$ [H,gu ] \le H_j$$
for every $u \in N$. Recall that $H_j$  is finite and central in $G$. We pass to the quotient $G/H_j$ and
without loss of generality we assume that $[H,gu ]=1$ for every $u \in N$. It follows that $[H,N]=1$. 
 Since $N$ is open, we can choose finitely many elements $g_1, \ldots , g_n$ that generate $G$ modulo $N$. 
%Let $G=\langle g_1, \ldots , g_n, N\rangle$. We
 Then we have 
$$[H,G]= \prod_{i=1}^n [H,g_i].$$
Since the subgroups $[H,g_i]$ are finite, so is $w(G)$. The proof is complete.
   \end{proof}

% ------------------------------------------------------------------------
\end{document}